\documentclass{article}

\usepackage[centertags]{amsmath}
\usepackage{hyperref}
\usepackage{amsfonts}
\usepackage{amssymb}
\usepackage{amsthm}
\usepackage{newlfont}
\usepackage{amscd}
\usepackage{amsmath,amscd}

\usepackage[curve,arrow,cmactex,matrix]{xy}
\usepackage{verbatim}

\usepackage{color}
\usepackage{eucal}

\newcommand{\C}{\mathcal{C}}
\newcommand{\F}{\mathcal{F}}
\newcommand{\OF}{\mathcal{OF}}
\newcommand{\OI}{\mathcal{OI}}

\newcommand{\Ass}{\mathcal{A}ss}
\newcommand{\UA}{\mathcal{UA}}
\newcommand{\OO}{\mathcal{O}}

\newcommand{\D}{\mathcal{D}}
\newcommand{\B}{\mathcal{B}}
\newcommand{\G}{\mathcal{G}}
\newcommand{\E}{\mathcal{E}}

\newcommand{\ZZ}{\mathbb{Z}}

\newcommand{\A}{\mathcal{A}}

\newcommand{\RR}{\mathbb{R}}

\def\alp{{\alpha}}

\def\sig{{\sigma}}
\def\vphi{{\varphi}}

\def\Gam{{\Gamma}}

\def\Lam{{\Lambda}}

\def\vphi{{\varphi}}

\newtheorem{thm}{Theorem}[section]

\newtheorem{cor}[thm]{Corollary}

\newtheorem{lem}[thm]{Lemma}

\newtheorem{example}{Example}
\theoremstyle{definition}
\newtheorem{define}[thm]{Definition}
\theoremstyle{remark}
\newtheorem{rem}[thm]{Remark}

\DeclareMathOperator{\Fun}{Fun}

\DeclareMathOperator{\Cat}{Cat}

\DeclareMathOperator{\Bord}{Bord}

\DeclareMathOperator{\ev}{ev}
\DeclareMathOperator{\coev}{coev}

\DeclareMathOperator{\Map}{Map}

\DeclareMathOperator{\Grp}{Grp}
\DeclareMathOperator{\lax}{lax}
\DeclareMathOperator{\nd}{nd}
\DeclareMathOperator{\qu}{qu}

\DeclareMathOperator{\Alg}{Alg}
\DeclareMathOperator{\act}{act}
\DeclareMathOperator{\ori}{or}
\DeclareMathOperator{\un}{un}
\DeclareMathOperator{\sur}{sur}
\DeclareMathOperator{\Groth}{Groth}
\DeclareMathOperator{\bg}{big}
\DeclareMathOperator{\df}{def}
\DeclareMathOperator{\Op}{Op}
\DeclareMathOperator{\Env}{Env}
\DeclareMathOperator{\N}{N}
\DeclareMathOperator{\nun}{nu}

\def\lrar{\longrightarrow}

\def\hrar{\hookrightarrow}

\def\x{\stackrel}

\def \mcal{\mathcal}
\def \ovl{\overline}
\def \what{\widehat}
\def \wtl{\widetilde}

\DeclareFontEncoding{OT2}{}{} 
\DeclareTextFontCommand{\textcyr}{\fontencoding{OT2}\fontfamily{wncyr}\fontseries{m}\fontshape{n}\selectfont}

\title{ The Cobordism Hypothesis in Dimension $1$}

\author{Yonatan Harpaz}

\begin{document}
\maketitle

\begin{abstract}
In~\cite{lur1} Lurie published an expository article outlining a proof for a higher version of the cobordism hypothesis conjectured by Baez and Dolan in~\cite{bd}. In this note we give a proof for the 1-dimensional case of this conjecture. The proof follows most of the outline given in~\cite{lur1}, but differs in a few crucial details. In particular, the proof makes use of the theory of quasi-unital $\infty$-categories as developed by the author in~\cite{har}.
\end{abstract}

\tableofcontents

\section{ Introduction }

Let $\B^{\ori}_1$ denote the $1$-dimensional oriented cobordism $\infty$-category, i.e. the symmetric monoidal $\infty$-category whose objects are oriented $0$-dimensional closed manifolds and whose morphisms are oriented $1$-dimensional cobordisms between them.

Let $\D$ be a symmetric monoidal $\infty$-category with duals. The $1$-dimensional cobordism hypothesis concerns the $\infty$-category
$$ \Fun^{\otimes}(\B^{\ori}_1,\D) $$
of symmetric monoidal functors $\vphi: \B^{\ori}_1 \lrar \D$. If $X_+ \in \B^{\ori}_1$ is the object corresponding to a point with positive orientation then the evaluation map $Z \mapsto Z(X_+)$ induces a functor
$$ \Fun^{\otimes}(\B^{\ori}_1,\D) \lrar \D $$

It is not hard to show that since $\B^{\ori}_1$ has duals the $\infty$-category $\Fun^{\otimes}(\B^{\ori}_1,\D)$ is in fact an $\infty$-groupoid, i.e. every natural transformation between two functors $F,G: \B^{\ori}_1 \lrar \D$ is a natural equivalence. This means that the evaluation map $Z \mapsto Z(X_+)$ actually factors through a map
$$ \Fun^{\otimes}(\B^{\ori}_1,\D) \lrar \wtl{\D} $$
where $\wtl{\D}$ is the maximal $\infty$-groupoid of $\D$. The cobordism hypothesis then states

\begin{thm}\label{cobordism-hypothesis}
The evaluation map
$$ \Fun^{\otimes}(\B^{\ori}_1,\D) \lrar \wtl{\D} $$
is an equivalence of $\infty$-categories.
\end{thm}

\begin{rem}
From the consideration above we see that we could have written the cobordism hypothesis as an equivalence
$$ \wtl{\Fun}^{\otimes}(\B^{\ori}_1,\D) \x{\simeq}{\lrar} \wtl{\D} $$
where $\wtl{\Fun}^{\otimes}(\B^{\ori}_1,\D)$ is the maximal $\infty$-groupoid of $\Fun^{\otimes}(\B^{\ori}_1,\D)$ (which in this case happens to coincide with $\Fun^{\otimes}(\B^{\ori}_1,\D)$). This $\infty$-groupoid is the fundamental groupoid of the space of maps from $\B^{\ori}_1$ to $\D$ in the $\infty$-category $\Cat^{\otimes}$ of symmetric monoidal $\infty$-categories.
\end{rem}

In his paper~\cite{lur1} Lurie gives an elaborate sketch of proof for a higher dimensional generalization of the $1$-dimensional cobordism hypothesis. For this one needs to generalize the notion of $\infty$-categories to $(\infty,n)$-categories. The strategy of proof described in~\cite{lur1} is inductive in nature. In particular in order to understand the $n=1$ case, one should start by considering the $n=0$ case.

Let $\B^{\un}_0$ be the $0$-dimensional unoriented cobordism category, i.e. the objects of $\B^{\un}_0$ are $0$-dimensional closed manifolds (or equivalently, finite sets) and the morphisms are diffeomorphisms (or equivalently, isomorphisms of finite sets). Note that $\B^{\un}_0$ is a (discrete) $\infty$-groupoid.

Let $X \in \B^{\un}_0$ be the object corresponding to one point. Then the $0$-dimensional cobordism hypothesis states that $\B^{\un}_0$ is in fact the free $\infty$-groupoid (or $(\infty,0)$-category) on one object, i.e. if  $\G$ is any other $\infty$-groupoid then the evaluation map $Z \mapsto Z(X)$ induces an equivalence of $\infty$-groupoids

$$ \Fun^{\otimes}(\B^{\un}_0,\G) \x{\simeq}{\lrar} \G $$
\begin{rem}
At this point one can wonder what is the justification for considering non-oriented manifolds in the $n=0$ case oriented ones in the $n=1$ case. As is explained in~\cite{lur1} the desired notion when working in the $n$-dimensional cobordism $(\infty,n)$-category is that of \textbf{$n$-framed} manifolds. One then observes that $0$-framed $0$-manifolds are unoriented manifolds, while taking $1$-framed $1$-manifolds (and $1$-framed $0$-manifolds) is equivalent to taking the respective manifolds with orientation.
\end{rem}

Now the $0$-dimensional cobordism hypothesis is not hard to verify. In fact, it holds in a slightly more general context - we do not have to assume that $\G$ is an $\infty$-groupoid. In fact, if $\G$ is \textbf{any symmetric monoidal $\infty$-category} then the evaluation map induces an equivalence of $\infty$-categories
$$ \Fun^{\otimes}(\B^{\un}_0,\G) \x{\simeq}{\lrar} \G $$
and hence also an equivalence of $\infty$-groupoids
$$ \wtl{\Fun}^{\otimes}(\B^{\un}_0,\G) \x{\simeq}{\lrar} \wtl{\G} $$

Now consider the under-category $\Cat^{\otimes}_{\B^{\un}_0/}$ of symmetric monoidal $\infty$-categories $\D$ equipped with a functor $\B^{\un}_0 \lrar \D$. Since $\B^{\un}_0$ is free on one generator this category can be identified with the  $\infty$-category of \textbf{pointed} symmetric monoidal $\infty$-categories, i.e. symmetric monoidal $\infty$-categories with a chosen object. We will often not distinguish between these two notions.

Now the point of positive orientation $X_+ \in \B^{\ori}_1$ determines a functor $\B^{\un}_0 \lrar \B^{\ori}_1$, i.e. an object in $\Cat^{\otimes}_{\B^{\un}_0/}$, which we shall denote by $\B^+_1$. The $1$-dimensional coborodism hypothesis is then equivalent to the following statement:
\begin{thm}\label{0-to-1}[Cobordism Hypothesis $0$-to-$1$]
Let $\D \in \Cat^{\otimes}_{\B^{\un}_0 /}$ be a pointed symmetric monoidal $\infty$-category with duals. Then the $\infty$-groupoid
$$ \wtl{\Fun}^{\otimes}_{\B^{\un}_0 /}(\B^+_1,\D) $$
is \textbf{contractible}.
\end{thm}

Theorem~\ref{0-to-1} can be considered as the inductive step from the $0$-dimensional cobordism hypothesis to the $1$-dimensional one. Now the strategy outlines in~\cite{lur1} proceeds to bridge the gap between $\B^{\un}_0$ to $\B^{\ori}_1$ by considering an intermediate $\infty$-category
$$ \B^{\un}_0 \hrar \B^{\ev}_1 \hrar \B^{\ori}_1 $$
This intermediate $\infty$-category is defined in~\cite{lur1} in terms of framed functions and index restriction. However in the $1$-dimensional case one can describe it without going into the theory of framed functors. In particular we will use the following definition:
\begin{define}
Let $\iota: \B^{\ev}_1 \hrar \B^{\ori}_1$ be the subcategory containing all objects and only the cobordisms $M$ in which every connected component $M_0 \subseteq M$ is either an identity segment or an evaluation segment.
\end{define}

Let us now describe how to bridge the gap between $\B^{\un}_0$ and $\B^{\ev}_1$. Let $\D$ be an $\infty$-category with duals and let
$$ \vphi:\B^{\ev}_1 \lrar \D $$
be a symmetric monoidal functor. We will say that $\vphi$ is \textbf{non-degenerate} if for each $X \in \B^{\ev}_1$ the map
$$ \vphi(\ev_X): \vphi(X) \otimes \vphi(\check{X}) \simeq \vphi(X \otimes \check{X}) \lrar \vphi(1) \simeq 1 $$
is \textbf{non-degenerate}, i.e. identifies $\vphi(\check{X})$ with a dual of $\vphi(X)$. We will denote by
$$ \Cat^{\nd}_{\B^{\ev}_1 /} \subseteq \Cat^{\otimes}_{\B^{\ev}_1 /} $$
the full subcategory spanned by objects $\vphi: \B^{\ev}_1 \lrar \D$ such that $\D$ has duals and $\vphi$ is non-degenerate.

Let $X_+ \in \B^{\ev}_1$ be the point with positive orientation. Then $X_+$ determines a functor
$$ \B^{\un}_0 \lrar \B^{\ev}_1 $$
The restriction map $\vphi \mapsto \vphi|_{\B^{\un}_0}$ then induces a functor
$$ \Cat^{\nd}_{\B^{\ev}_1 /} \lrar \Cat^{\otimes}_{\B^{\un}_0 /} $$

Now the gap between $\B^{\ev}_1$ and $\B^{\un}_0$ can be climbed using the following lemma (see~\cite{lur1}):
\begin{lem}\label{0-to-1-ev}
The functor
$$ \Cat^{\nd}_{\B^{\ev}_1 /} \lrar \Cat^{\otimes}_{\B^{\un}_0 /} $$
is fully faithful.
\end{lem}
\begin{proof}
First note that if $F:\D \lrar \D'$ is a symmetric monoidal functor where $\D,\D'$ have duals and $\vphi: \B^{\ev}_1 \lrar \D$ is non-degenerate then $f \circ \vphi$ will be non-degenerate as well. Hence it will be enough to show that if $\D$ has duals then the restriction map induces an equivalence between the $\infty$-groupoid of non-degenerate symmetric monoidal functors
$$ \B^{\ev}_1 \lrar \D $$
and the $\infty$-groupoid of symmetric monoidal functors
$$ \B^{\un}_0 \lrar \D $$

Now specifying a non-degenerate functor
$$ \B^{\ev}_1 \lrar \D $$
is equivalent to specifying a pair of objects $D_+,D_- \in \D$ (the images of $X_+,X_-$ respectively) and a non-degenerate morphism
$$ e: D_+ \otimes D_- \lrar 1 $$
which is the image of $\ev_{X_+}$. Since $\D$ has duals the $\infty$-groupoid of triples $(D_+,D_-,e)$ in which $e$ is non-degenerate is equivalent to the $\infty$-groupoid of triples $(D_+,\check{D}_-,f)$ where $f: D_+ \lrar \check{D}_-$ is an equivalence. Hence the forgetful map $(D_+,D_-,e) \mapsto D_+$ is an equivalence.
\end{proof}

Now consider the natural inclusion $\iota: \B^{\ev}_1 \lrar \B^{\ori}_1$ as an object in $\Cat^{\nd}_{\B^{\ev}_1 /}$. Then by Lemma~\ref{0-to-1-ev} we see that the $1$-dimensional cobordism hypothesis will be established once we make the following last step:

\begin{thm}[Cobordism Hypothesis - Last Step]\label{cobordism-last-step}
Let $\D$ be a symmetric monoidal $\infty$-category with duals and let $\vphi: \B^{\ev}_1 \lrar \D$ be a \textbf{non-degenerate} functor. Then the $\infty$-groupoid
$$ \wtl{\Fun}^{\otimes}_{\B^{\ev}_1 /}(\B^{\ori}_1,\D) $$
is contractible.
\end{thm}

Note that since $\B^{\ev}_1 \lrar \B^{\ori}_1$ is essentially surjective all the functors in
$$ \wtl{\Fun}^{\otimes}_{\B^{\ev}_1 /}(\B^{\ori}_1,\D) $$
will have the same essential image of $\vphi$. Hence it will be enough to prove for the claim for the case where $\vphi: \B^{\ev}_1 \lrar \D$ is \textbf{essentially surjective}. We will denote by
$$ \Cat^{\sur}_{\B^{\ev}_1 /} \subseteq \Cat^{\nd}_{\B^{\ev}_1 /} $$
the full subcategory spanned by essentially surjective functors $\vphi: \B^{\ev}_1 \lrar \D$. Hence we can phrase Theorem~\ref{cobordism-last-step} as follows:

\begin{thm}[Cobordism Hypothesis - Last Step 2]\label{cobordism-last-step-2}
Let $\D$ be a symmetric monoidal $\infty$-category with duals and let $\vphi: \B^{\ev}_1 \lrar \D$ be an \textbf{essentially surjective non-degenerate} functor. Then the space of maps
$$ \Map_{\Cat^{\sur}_{\B^{\ev}_1 /}}(\iota,\vphi) $$
is contractible.
\end{thm}

The purpose of this paper is to provide a formal proof for this last step. This paper is constructed as follows. In \S~\ref{s-qu-cobordism} we prove a variant of Theorem~\ref{cobordism-last-step-2} which we call the quasi-unital cobordism hypothesis (Theorem~\ref{qu-cobordism}). Then in \S~\ref{s-from-qu-to-regular} we explain how to deduce Theorem~\ref{cobordism-last-step-2} from Theorem~\ref{qu-cobordism}. Section \S~\ref{s-from-qu-to-regular} relies on the notion of \textbf{quasi-unital $\infty$-categories} which is developed rigourously in~\cite{har} (however \S~\ref{s-qu-cobordism} is completely independent of~\cite{har}).

\section{ The Quasi-Unital Cobordism Hypothesis }\label{s-qu-cobordism}

Let $\vphi: \B^{\ev}_1 \lrar \D$ be a non-degenerate functor and let $\Grp_\infty$ denote the $\infty$-category of $\infty$-groupoids. We can define a lax symmetric functor $M_\vphi: \B^{\ev}_1 \lrar \Grp_{\infty}$ by setting
$$ M_\vphi(X) = \Map_{\D}(1,\vphi(X)) $$
We will refer to $M_\vphi$ as the \textbf{fiber functor} of $\vphi$. Now if $\D$ has duals and $\vphi$ is non-degenerate, then one can expect this to be reflected in $M_\vphi$ somehow. More precisely, we have the following notion:
\begin{define}
Let $M: \B^{\ev}_1 \lrar \Grp_{\infty}$ be a lax symmetric monoidal functor. An object $Z \in M(X \otimes \check{X})$ is called \textbf{non-degenerate} if for each object $Y \in \B^{\ev}_1$ the natural map
$$ M(Y \otimes \check{X}) \x{Id \times Z}{\lrar} M(Y \otimes \check{X}) \times M(X \otimes \check{X}) \lrar M(Y \otimes \check{X} \otimes X \otimes \check{X}) \x{M(Id \otimes \ev \otimes Id)}{\lrar} M(Y \otimes \check{X}) $$
is an equivalence of $\infty$-groupoids.
\end{define}

\begin{rem}\label{uniqueness}
If a non-degenerate element $Z \in M(X \otimes \check{X})$ exists then it is unique up to a (non-canonical) equivalence.
\end{rem}
\begin{example}\label{unit}
Let $M: \B^{\ev}_1 \lrar \Grp_{\infty}$ be a lax symmetric monoidal functor. The lax symmetric structure of $M$ includes a structure map $1_{\Grp_{\infty}} \lrar M(1)$ which can be described by choosing an object $Z_1 \in M(1)$. The axioms of lax monoidality then ensure that $Z_1$ is non-degenerate.
\end{example}

\begin{define}
A lax symmetric monoidal functor $M: \B^{\ev}_1 \lrar \Grp_{\infty}$ will be called \textbf{non-degenerate} if for each object $X \in \B^{\ev}_1$ there exists a non-degenerate object $Z \in M(X \otimes \check{X})$.
\end{define}

\begin{define}
Let $M_1,M_2: \B^{\ev}_1 \lrar \Grp_{\infty}$ be two non-degenerate lax symmetric monoidal functors. A lax symmetric natural transformation $T: M_1 \lrar M_2$ will be called \textbf{non-degenerate} if for each object $X \in \Bord^{\ev}$ and each non-degenerate object $Z \in M(X \otimes \check{X})$ the objects $T(Z) \in M_2(X \otimes \check{X})$ is non-degerate.
\end{define}

\begin{rem}
From remark~\ref{uniqueness} we see that if $T(Z) \in M_2(X \otimes \check{X})$ is non-degenerate for \textbf{at least one} non-degenerate $Z \in M_1(X \otimes \check{X})$ then it will be true for all non-degenerate $Z \in M_1(X \otimes \check{X})$.
\end{rem}

Now we claim that if $\D$ has duals and $\vphi: \B^{\ev}_1 \lrar \D$ is non-degenerate then the fiber functor $M_\vphi$ will be non-degenerate: for each object $X \in \B^{\ev}_1$ there exists a coevaluation morphism
$$ \coev_{\vphi(X)}: 1 \lrar \vphi(X) \otimes \vphi(\check{X}) \simeq \vphi(X \otimes \check{X}) $$
which determines an element in $Z_X \in M_\vphi(X \otimes \check{X})$. It is not hard to see that this element is non-degenerate.

Let $\Fun^{\lax}(\B^{\ev}_1,\Grp_{\infty})$  denote the $\infty$-category of lax symmetric monoidal functors $\B^{\ev}_1 \lrar \Grp_{\infty}$ and by
$$ \Fun_{\nd}^{\lax}(\B^{\ev}_1,\Grp_{\infty}) \subseteq \Fun^{\lax}(\B^{\ev}_1,\Grp_{\infty}) $$
the subcategory spanned by non-degenerate functors and non-degenerate natural transformations. Now the construction $\vphi \mapsto M_\vphi$ determines a functor

$$ \Cat^{\nd}_{\B^{\ev}_1 /} \lrar \Fun_{\nd}^{\lax}(\B^{\ev}_1,\Grp_{\infty}) $$
In particular if $\vphi: \B^{\ev}_1 \lrar \C$ and $\psi: \B^{\ev}_1 \lrar \D$ are non-degenerate then any functor $T:\C \lrar \D$ under $\B^{\ev}_1$ will induce a non-degenerate natural transformation
$$ T_*: M_{\vphi} \lrar M_{\psi} $$

The rest of this section is devoted to proving the following result, which we call the "quasi-unital cobordism hypothesis":

\begin{thm}[Cobordism Hypothesis - Quasi-Unital] \label{qu-cobordism}
Let $\D$ be a symmetric monoidal $\infty$-category with duals, let $\vphi: \B^{\ev}_1 \lrar \D$ be a non-degenerate functor and let $\iota: \B^{\ev}_1 \hrar \B^{\ori}_1$ be the natural inclusion. Let $M_\iota,M_\vphi \in \Fun^{\lax}_{\nd}$ be the corresponding fiber functors. Them the space of maps
$$ \Map_{\Fun^{\lax}_{\nd}}(M_\iota, M_\vphi) $$
is contractible.
\end{thm}

\begin{proof}
We start by transforming the lax symmetric monoidal functors $M_\iota,M_\vphi$ to \textbf{left fibrations} over $\B^{\ev}_1$ using the symmetric monoidal analogue of Grothendieck's construction, as described in~\cite{lur1}, page $67-68$.

Let $M: \B \lrar \Grp_\infty$ be a lax symmetric monoidal functor. We can construct a symmetric monoidal $\infty$-category $\Groth(\B,M)$ as follows:
\begin{enumerate}
\item
The objects of $\Groth(\B,M)$ are pairs $(X, \eta)$ where $X \in \B$ is an object and $\eta$ is an object of $M(X)$.
\item
The space of maps from $(X,\eta)$ to $(X',\eta')$ in $\Groth(\B,M)$ is defined to be the classifying space of the $\infty$-groupoid of pairs $(f,\alp)$ where $f: X \lrar X'$ is a morphism in $B$ and $\alp: f_*\eta \lrar \eta$ is a morphism in $M(X')$. Composition is defined in a straightforward way.
\item
The symmetric monoidal structure on $\Groth(\B,M)$ is obtained by defining
$$ (X,\eta) \otimes (X',\eta') = (X \otimes X',\beta_{X,Y}(\eta \otimes \eta')) $$
where $\beta_{X,Y}: M(X) \times M(Y) \lrar M(X \otimes Y)$ is given by the lax symmetric structure of $M$.
\end{enumerate}

The forgetful functor $(X,\eta) \mapsto X$ induces a \textbf{left fibration}
$$ \Groth(\B,M) \lrar \B $$

\begin{thm}\label{unstraightening}
The association $M \mapsto \Groth(\B,M)$ induces an equivalence between the $\infty$-category of lax-symmetric monoidal functors $\B \lrar \Grp_\infty$ and the full subcategory of the over $\infty$-category
$ \Cat^{\otimes}_{/\B} $
spanned by left fibrations.
\end{thm}
\begin{proof}
This follows from the more general statement given in~\cite{lur1} Proposition $3.3.26$. Note that any map of left fibrations over $\B$ is in particular a map of coCartesian fibrations because if $p: \C \lrar \B$ is a left fibration then any edge in $\C$ is $p$-coCartesian.
\end{proof}

\begin{rem}
Note that if $\C \lrar \B$ is a left fibration of symmetric monoidal $\infty$-categories and $\A \lrar \B$ is a symmetric monoidal functor then the $\infty$-category
$$ \Fun^{\otimes}_{/ \B}(\A,\C) $$
is actually an \textbf{$\infty$-groupoid}, and by Theorem~\ref{unstraightening} is equivalent to the $\infty$-groupoid of lax-monoidal natural transformations between the corresponding lax monoidal functors from $\B$ to $\Grp_\infty$.
\end{rem}

Now set
$$ \F_\iota \x{\df}{=} \Groth(\B^{\ev}_1,M_{\iota}) $$
$$ \F_\vphi \x{\df}{=} \Groth(\B^{\ev}_1,M_{\vphi}) $$

Let
$$ \Fun^{\nd}_{/\B^{\ev}_1}(\F_{\iota},\F_{\vphi}) \subseteq \Fun^{\otimes}_{/\B^{\ev}_1}(\F_{\iota},\F_{\vphi}) $$
denote the full sub $\infty$-groupoid of functors which correspond to \textbf{non-degenerate} natural transformations
$$ M_\iota \lrar M_\vphi $$
under the Grothendieck construction. Note that $\Fun^{\nd}_{/\B^{\ev}_1}(\F_{\iota},\F_{\vphi})$ is a union of connected components of the $\infty$-groupoid $\Fun^{\otimes}_{/\B^{\ev}_1}(\F_{\iota},\F_{\vphi})$.

We now need to show that the $\infty$-groupoid
$$ \Fun^{\nd}_{/\B^{\ev}_1}(\F_{\iota},\F_{\vphi}) $$
is contractible.

Unwinding the definitions we see that the objects of $\F_{\iota}$ are pairs $(X,M)$ where $X \in \B^{\ev}_1$ is a $0$-manifold and $M \in \Map_{\B^{\ori}_1}(\emptyset,X)$ is a cobordism from $\emptyset$ to $X$. A morphism in $\vphi$ from $(X,M)$ to $(X',M')$ consists of a morphism in $\B^{\ev}_1$
$$ N:X \lrar X' $$
and a diffeomorphism
$$ T:M \coprod_{X} N \cong M' $$
respecting $X'$. Note that for each $(X,M) \in \F_{\iota}$ we have an identification $X \simeq \partial M$. Further more the space of morphisms from $(\partial M,M)$ to $(\partial M',M')$ is \textbf{homotopy equivalent to the space of orientation-preserving $\pi_0$-surjective embeddings of $M$ in $M'$} (which are not required to respect the boundaries in any way).

Now in order to analyze the symmetric monoidal $\infty$-category $\F_\iota$ we are going to use the theory of \textbf{$\infty$-operads}, as developed in~\cite{lur2}. Recall that the category $\Cat^{\otimes}$ of symmetric monoidal $\infty$-categories admits a forgetful functor
$$ \Cat^{\otimes} \lrar \Op^{\infty} $$
to the $\infty$-category of \textbf{$\infty$-operads}. This functor has a left adjoint
$$ \Env: \Op^{\infty} \lrar \Cat^{\otimes} $$
called the \textbf{monoidal envelope} functor (see~\cite{lur2} \S $2.2.4$). In particular, if $\mcal{C}^{\otimes}$ is an $\infty$-operad and $\D$ is a symmetric monoidal $\infty$-category with corresponding $\infty$-operad $\D^{\otimes} \lrar \N(\Gam_*)$ then there is an \textbf{equivalence of $\infty$-categories}
$$ \Fun^{\otimes}(\Env(\C^{\otimes}),\D) \simeq \Alg_{\mcal{C}}(\D^{\otimes}) $$
Where $\Alg_{\mcal{C}}\left(\D^{\otimes}\right) \subseteq \Fun_{/\N(\Gam_*)}(\mcal{C}^{\otimes},\D^{\otimes})$ denotes the full subcategory spanned by $\infty$-operad maps (see Proposition $2.2.4.9$ of~\cite{lur2}).

Now observing the definition of monoidal envelop (see Remark $2.2.4.3$ in~\cite{lur2}) we see that $\F_{\iota}$ is equivalent to the monoidal envelope of a certain simple $\infty$-operad
$$ F_\iota \simeq \Env\left(\OF^{\otimes}\right) $$
which can be described as follows: the underlying $\infty$-category $\OF$ of $\OF^{\otimes}$ is the $\infty$-category of \textbf{connected} $1$-manifolds (i.e. either the segment or the circle) and the morphisms are \textbf{orientation-preserving embeddings} between them. The (active) $n$-to-$1$ operations of $\OF$ (for $n\geq 1$) from $(M_1,...,M_n)$ to $M$ are the orientation-preserving embeddings
$$ M_1 \coprod ... \coprod M_n \lrar M $$
and there are no $0$-to-$1$ operations.

Now observe that the induced map $\OF^{\otimes} \lrar (\B^{\ev}_1)^{\infty}$ is a fibration of $\infty$-operads. We claim that $\F_{\iota}$ is not only the enveloping symmetric monoidal $\infty$-category of $\OF^{\otimes}$, but that $\F_{\iota} \lrar \B^{\ev}_1$ is the enveloping \textbf{left fibration} of $\OF \lrar \B^{\ev}_1$. More precisely we claim that for any left fibration $\D \lrar \B^{\ev}_1$ of symmetric monoidal $\infty$-categories the natural map
$$ \Fun^{\otimes}_{/\B^{\ev}_1}\left(F_{\iota},\D\right) \lrar \Alg_{\OF / \B^{\ev}_1}(\D^{\otimes}) $$
is an equivalence if $\infty$-groupoids (where both terms denote mapping objects in the respective \textbf{over-categories}). This is in fact not a special property of $F_{\iota}$:
\begin{lem}\label{left-envelope}
Let $\OO$ be a symmetric monoidal $\infty$-category with corresponding $\infty$-operad $\OO^{\otimes} \lrar \N(\Gam_*)$ and let $p:\C^{\otimes} \lrar \OO^{\otimes}$ be a fibration of $\infty$-operads such that the induced map
$$ \ovl{p}:\Env\left(\C^{\otimes}\right) \lrar \OO $$
is a left fibration. Let $\D \lrar \OO$ be some other left fibration of symmetric monoidal categories. Then the natural map
$$ \Fun^{\otimes}_{/\OO}\left(\Env\left(\C^{\otimes}\right),\D\right) \lrar \Alg_{\C / \OO}(\D^{\otimes}) $$
is an equivalence of $\infty$-categories. Further more both sides are in fact $\infty$-groupoids.
\end{lem}
\begin{proof}
Consider the diagram
$$ \xymatrix{
\Fun^{\otimes}(\Env\left(\C^{\otimes}\right),\D) \ar^{\simeq}[r]\ar[d] & \Alg_{\C}\left(\D^{\otimes}\right) \ar[d] \\
\Fun^{\otimes}(\Env\left(\C^{\otimes}\right),\OO) \ar^{\simeq}[r]       & \Alg_{\C}\left(\OO^{\otimes}\right) \\
}$$
Now the vertical maps are left fibrations and by adjunction the horizontal maps are equivalences. By~\cite{lur3} Proposition $3.3.1.5$ we get that the induced map on the fibers of $p$ and $\ovl{p}$ respectively
$$ \Fun^{\otimes}_{/\OO}\left(\Env\left(\C^{\otimes}\right),\D\right) \lrar \Alg_{\C / \OO}(\D^{\otimes}) $$
is a weak equivalence of $\infty$-groupoids.
\end{proof}

\begin{rem}
In~\cite{lur2} a relative variant $\Env_{\B^{\ev}_1}$ of $\Env$ is introduced which sends a fibration of $\infty$-operads $\C^{\otimes} \lrar (\B^{\ev}_1)^{\otimes}$ to its enveloping coCartesin fibration $\Env_{\OO}\left(\C^{\otimes}\right) \lrar \B^{\ev}_1$. Note that in our case the map
$$ \F_{\iota} \lrar \B^{\ev}_1 $$
is \textbf{not} the enveloping coCartesian fibration of $\OF^{\otimes} \lrar (\B^{\ev}_1)^{\otimes}$. However from Lemma~\ref{left-envelope} it follows that the map
$$ \xymatrix{
\F_{\iota} \ar[rr]\ar[dr] && \Env_{\B^{\ev}_1}\left(\OF^{\otimes}\right) \ar[dl] \\
& \B^{\ev}_1 & \\
}$$
is a \textbf{covariant equivalence} over $\B^{\ev}_1$, i.e. induces a weak equivalence of simplicial sets on the fibers (where the fibers on the left are $\infty$-groupoids and the fibers on the right are $\infty$-categories). This claim can also be verified directly by unwinding the definition of $\Env_{\B^{\ev}_1}\left(\OF^{\otimes}\right)$.
\end{rem}

Summing up the discussion so far we observe that we have a weak equivalence of $\infty$-groupoids
$$ \Fun^{\otimes}_{/\B^{\ev}_1}\left(\F_{\iota},\F_{\vphi}\right) \x{\simeq}{\lrar} \Alg_{\OF / \B^{\ev}_1}\left(\F_{\vphi}^{\otimes}\right) $$

Let
$$ \Alg^{\nd}_{\OF / \B^{\ev}_1}\left(\F_{\vphi}^{\otimes}\right) \subseteq \Alg_{\OF / \B^{\ev}_1}\left(\F_{\vphi}^{\otimes}\right) $$
denote the full sub $\infty$-groupoid corresponding to
$$ \Fun^{\nd}_{/\B^{\ev}_1}(\F_{\iota},\F_{\vphi}) \subseteq \Fun^{\otimes}_{/\B^{\ev}_1}(\F_{\iota},\F_{\vphi}) $$
under the adjunction. We are now reduced to prove that the $\infty$-groupoid
$$ \Alg^{\nd}_{\OF / \B^{\ev}_1}\left(\F_{\vphi}^{\otimes}\right) $$
is contractible.

Let $\OI^{\otimes} \subseteq \OF^{\otimes}$ be the full sub $\infty$-operad of $\OF^{\otimes}$ spanned by connected $1$-manifolds which are diffeomorphic to the segment (and all $n$-to-$1$ operations between them). In particular we see that $\OI^{\otimes}$ is equivalent to the \textbf{non-unital associative $\infty$-operad}.

We begin with the following theorem which reduces the handling of $\OF^{\otimes}$ to $\OI^{\otimes}$.

\begin{thm}\label{removing-circles}
Let $q:\C^{\otimes} \lrar \OO^{\otimes}$ be a left fibration of $\infty$-operads. Then the restriction map
$$ \Alg_{\OF / \OO}(\C^{\otimes}) \lrar \Alg_{\OI / \OO}(\C^{\otimes}) $$
is a weak equivalence.
\end{thm}
\begin{proof}
We will base our claim on the following general lemma:
\begin{lem}\label{free-algebra}
Let $\A^{\otimes} \lrar \B^{\otimes}$ be a map of $\infty$-groupoids and let $q:\C^{\otimes} \lrar \OO^{\otimes}$ be \textbf{left fibration} of $\infty$-operads. Suppose that for every object $B \in \B$, the category
$$ \F_B = \A^{\otimes}_{\act} \times_{\B^{\otimes}_{\act}} \B^{\otimes}_{/B} $$
is weakly contractible (see~\cite{lur2} for the terminology). Then the natural restriction map
$$ \Alg_{\A / \OO}(\C^{\otimes}) \lrar \Alg_{\B / \OO}(\C^{\otimes}) $$
is a weak equivalence.
\end{lem}
\begin{proof}
In~\cite{lur2} \S $3.1.3$ it is explained how under certain conditions the forgetful functor (i.e. restriction map)
$$ \Alg_{\A / \OO}(\C^{\otimes}) \lrar \Alg_{\B / \OO}(\C^{\otimes}) $$
admits a left adjoint, called the \textbf{free algebra functor}. Since $\C^{\otimes} \lrar \OO^{\otimes}$ is a left fibration both these $\infty$-categories are $\infty$-groupoids, and so any adjunction between them will be an equivalence. Hence it will suffice to show that the conditions for existence of left adjoint are satisfies in this case.

Since $q: \C^{\otimes} \lrar \OO^{\otimes}$ is a left fibration $q$ is \textbf{compatible with colimits indexed by weakly contractible diagrams} in the sense of~\cite{lur2} Definition $3.1.1.18$ (because weakly contractible colimits exists in every $\infty$-groupoid and are preserved by any functor between $\infty$-groupoids). Combining Corollary $3.1.3.4$ and Proposition $3.1.1.20$ of~\cite{lur2} we see that the desired free algebra functor exists.
\end{proof}

In view of Lemma~\ref{free-algebra} it will be enough to check that for every object $M \in \OF$ (i.e. every connected $1$-manifolds) the $\infty$-category
$$ \F_M \x{\df}{=} \OI^{\otimes}_{\act} \times_{\OF^{\otimes}_{\act}} \left(\OF^{\otimes}_{\act}\right)_{/M} $$
is weakly contractible.

Unwinding the definitions we see that the objects of $\F_M$ are tuples of $1$-manifolds $(M_1,...,M_n)$ ($n \geq 1$), such that each $M_i$ is diffeomorphic to a segment, together with an orientation preserving embedding
$$ f: M_1 \coprod ... \coprod M_n \hrar M $$
A morphisms in $\F_M$ from
$$ f: M_1 \coprod ... \coprod M_n \hrar M $$
to
$$ g: M_1' \coprod ... \coprod M_m' \hrar M $$
is a $\pi_0$-surjective orientation-preserving embedding
$$ T:M_1 \coprod ... \coprod M_n \lrar M_1' \coprod ... \coprod M_m' $$
together with an \textbf{isotopy} $g \circ T \sim f$.

Now when $M$ is the segment then $\F_M$ contains a terminal object and so is weakly contractible. Hence we only need to take care of the case of the circle $M=S^1$.

It is not hard to verify that the category $F_{S^1}$ is in fact discrete - the space of self isotopies of any embedding $f:M_1 \coprod ... \coprod M_n \hrar M $ is equivalent to the loop space of $S^1$ and hence discrete. In fact one can even describe $F_{S^1}$ in completely combinatorial terms. In order to do that we will need some terminology.

\begin{define}
Let $\Lam_{\infty}$ be the category whose objects correspond to the natural numbers $1,2,3,...$ and the morphisms from $n$ to $m$ are (weak) order preserving maps $f: \ZZ \lrar \ZZ$ such that $f(x+n) = f(x)+m$.
\end{define}
The category $\Lam_{\infty}$ is a model for the the universal fibration over the cyclic category, i.e., there is a left fibration $\Lam_\infty \lrar \Lam$ (where $\Lam$ is connes' cyclic category) such that the fibers are connected groupoids with a single object having automorphism group $\ZZ$ (or in other words circles). In particular the category $\Lam_{\infty}$ is known to be weakly contractible. See~\cite{kal} for a detailed introduction and proof (Lemma $4.8$).

Let $\Lam^{\sur}_{\infty}$ be the sub category of $\Lam_\infty$ which contains all the objects and only \textbf{surjective} maps between. It is not hard to verify explicitly that the map $\Lam^{\sur}_\infty \lrar \Lam_\infty$ is cofinal and so $\Lam^{\sur}_{\infty}$ is contractible as well. Now we claim that $F_{S^1}$ is in fact equivalent to $\Lam^{\sur}_{\infty}$.

Let $\Lam^{\sur}_{\bg}$ be the category whose objects are linearly ordered sets $S$ with an order preserving automorphisms $\sig: S \lrar S$ and whose morphisms are surjective order preserving maps which commute with the respective automorphisms. Then $\Lam^{\sur}_{\infty}$ can be considered as a full subcategory of $\Lam^{\sur}_{\bg}$ such that $n$ corresponds to the object $(\ZZ,\sig_n)$ where $\sig_n: \ZZ \lrar \ZZ$ is the automorphism $x \mapsto x+n$.

Now let $p:\RR \lrar S^1$ be the universal covering. We construct a functor $F_{S^1} \lrar \Lam^{\sur}_{\bg}$ as follows: given an object
$$ f: M_1 \coprod ... \coprod M_n \hrar S^1 $$
of $F_{S^1}$ consider the fiber product
$$ P = \left[M_1 \coprod ... \coprod M_n\right] \times_{S^1} \RR $$
note that $P$ is homeomorphic to an infinite union of segments and the projection
$$ P \lrar \RR $$
is injective (because $f$ is injective) giving us a well defined linear order on $P$. The automorphism $\sig: \RR \lrar \RR$ of $\RR$ over $S^1$ given by $x \mapsto x + 1$ gives an order preserving automorphism $\wtl{\sig}: P \lrar P$.

Now suppose that $((M_1,...,M_n),f)$ and $((M_1',...,M_m'),g)$ are two objects and we have a morphism between them, i.e. an embedding

$$ T:M_1 \coprod ... \coprod M_n \lrar M_1' \coprod ... \coprod M_m' $$
and an isotopy $\psi: g \circ T \sim f$. Then we see that the pair $(T,\psi)$ determine a well defined order preserving map
$$ \left[M_1 \coprod ... \coprod M_n\right] \times_{S^1} \RR \lrar \left[M_1' \coprod ... \coprod M_m'\right] \times_{S^1} \RR $$
which commutes with the respective automorphisms. Clearly we obtain in this way a functor $u:F_{S^1} \lrar \Lam^{\sur}_{\bg}$ whose essential image is the same as the essential image of $\Lam^{\sur}_\infty$. It is also not hard to see that $u$ is fully faithful. Hence $F_{S^1}$ is equivalent to $\Lam^{\sur}_\infty$ which is weakly contractible. This finishes the proof of the theorem.

\end{proof}
%
%
Let
$$ \Alg^{\nd}_{\OI / \B^{\ev}_1}\left(\F_{\vphi}^{\otimes}\right) \subseteq \Alg_{\OI / \B^{\ev}_1}\left(\F_{\vphi}^{\otimes}\right) $$
denote the full sub $\infty$-groupoid corresponding to the full sub $\infty$-groupoid
$$ \Alg^{\nd}_{\OF / \B^{\ev}_1}\left(\F_{\vphi}^{\otimes}\right) \subseteq \Alg_{\OF / \B^{\ev}_1}\left(\F_{\vphi}^{\otimes}\right) $$
under the equivalence of Theorem~\ref{removing-circles}.

Now the last step of the cobordism hypothesis will be complete once we show the following:
\begin{lem}\label{final-lemma}
The $\infty$-groupoid
$$ \Alg^{\nd}_{\OI / \B^{\ev}_1}\left(\F_{\vphi}^{\otimes}\right) $$
is contractible.
\end{lem}
\begin{proof}
Let
$$ q: p^*\F_{\vphi} \lrar \OI^{\otimes} $$
be the pullback of left fibration $\F_\vphi \lrar \B^{\ev}_1$ via the map $p: \OI^{\otimes} \lrar B^{\ev}_1$, so that $q$ is a left fibration as well. In particular, since $\OI^{\otimes}$ is the non-unital associative $\infty$-operad, we see that $q$ classifies an $\infty$-groupoid $q^{-1}(\OI)$ with a non-unital monoidal structure. Unwinding the definitions one sees that this $\infty$-groupoid is the fundamental groupoid of the space
$$ \Map_{\C}(1,\vphi(X_+) \otimes \vphi(X_-)) $$
where $X_+,X_- \in \B^{\ev_1}$ are the points with positive and negative orientations respectively. The monoidal structure sends a pair of maps
$$ f,f': 1 \lrar \vphi(X_+) \otimes \vphi(X_-) $$
to the composition
$$ 1 \x{f \otimes f'}{\lrar} \left[\vphi(X_+) \otimes \vphi(X_-)\right] \otimes \left[\vphi(X_+) \otimes \vphi(X_-)\right] \x{\simeq}{\lrar} $$
$$ \vphi(X_+) \otimes \left[\vphi(X_-) \otimes \vphi(X_+)\right] \otimes \vphi(X_-) \x{Id \otimes \vphi(\ev) \otimes Id}{\lrar} \vphi(X_+) \otimes \vphi(X_-) $$
Since $\C$ has duals we see that this monoidal $\infty$-groupoid is equivalent to the fundamental $\infty$-groupoid of the space
$$ \Map_{\C}(\vphi(X_+),\vphi(X_+)) $$
with the monoidal product coming from \textbf{composition}.

Now
$$ \Alg_{\OI / \B^{\ev}_1}(\F_{\vphi}) \simeq \Alg_{\OI / \OI}(p^*\F_{\vphi}) $$
classifies $\OI^{\otimes}$-algebra objects in $p^*\F_{\vphi}$, i.e. non-unital algebra objects in
$$ \Map_{\C}(\vphi(X_+),\vphi(X_+)) $$
with respect to composition. The full sub $\infty$-groupoid
$$ \Alg^{\nd}_{\OI / \B^{\ev}_1}(\F_{\vphi}) \subseteq \Alg_{\OI / \B^{\ev}_1}(\F_{\vphi}) $$
will then classify non-unital algebra objects $A$ which correspond to \textbf{self equivalences}
$$ \vphi(X_+) \lrar \vphi(X_+)  $$
It is left to prove the following lemma:
\begin{lem}
Let $\C$ be an $\infty$-category. Let $X \in \C$ be an object and let $\E_X$ denote the $\infty$-groupoid of self equivalences $u: X \lrar X$ with the monoidal product induced from composition. Then the $\infty$-groupoid of non-unital algebra objects in $\E_X$ is contractible.
\end{lem}
\begin{proof}
Let $\Ass_{\nun}$ denote the non-unital associative $\infty$-operad. The identity map $\Ass_{\nun} \lrar \Ass_{\nun}$ which is in particular a left fibration of $\infty$-operads classifies the terminal non-unital monoidal $\infty$-groupoid $\A$ which consists of single automorphismless idempotent object $a \in \A$. The non-unital algebra objects in $\E_X$ are then classified by non-unital lax monoidal functors
$$ \A \lrar \E_X $$
Since $\E_X$ is an $\infty$-groupoid this is same as non-unital monoidal functors (without the lax)
$$ \A \lrar \E_X $$
Now the forgetful functor from unital to non-unital monoidal $\infty$-groupoids has a left adjoint. Applying this left adjoint to $\A$ we obtain the $\infty$-groupoid $\UA$ with two automorphismless objects
$$ \UA = \{1,a\} $$
such that $1$ is the unit of the monoidal structure and $a$ is an idempotent object.

Hence we need to show that the $\infty$-groupoids of monoidal functors
$$ \UA \lrar \E_X $$
is contractible. Now given a monoidal $\infty$-groupoid $\G$ we can form the $\infty$-category $\B(\G)$ having a single object with endomorphism space $\G$ (the monoidal structure on $\G$ will then give the composition structure). This construction determines a fully faithful functor from the $\infty$-category of monoidal $\infty$-groupoids and the $\infty$-category of pointed $\infty$-categories (see~\cite{lur1} Remark $4.4.6$ for a much more general statement). In particular it will be enough to show that the $\infty$-groupoid of \textbf{pointed functors}
$$ \B(\UA) \lrar \B(\E_X) $$
is contractible. Since $\B(\E_X)$ is an $\infty$-groupoid it will be enough to show that $\B(\UA)$ is weakly contractible.

Now the nerve $\N\B(\UA)$ of $\B(\UA)$ is the simplicial set in which for each $n$ there exists a single \textbf{non-degenerate} $n$-simplex $\sig_n \in \N\B(\UA)_n$ such that $d_i(\sig_n) = \sig_{n-1}$ for all $i=0,...,n$. By Van-Kampen it follows that $\N\B(\UA)$ is simply connected and by direct computation all the homology groups vanish.
\end{proof}
This finishes the proof of Lemma~\ref{final-lemma}.
\end{proof}

This finishes the proof of Theorem~\ref{qu-cobordism}.
\end{proof}

\section{ From Quasi-Unital to Unital Cobordism Hypothesis }\label{s-from-qu-to-regular}

In this section we will show how the quasi-unital cobordism hypothesis (Theorem~\ref{qu-cobordism}) implies the last step in the proof of the $1$-dimensional cobordism hypothesis (Theorem~\ref{cobordism-last-step-2}).

Let $M: \B^{\ev}_1 \lrar \Grp_{\infty}$ be a non-degenerate lax symmetric monoidal functor. We can construct a pointed \textbf{non-unital} symmetric monoidal $\infty$-category $\C_M$ as follows:
\begin{enumerate}
\item
The objects of $\C_M$ are the objects of $\B^{\ev}_1$. The marked point is the object $X_+$.
\item
Given a pair of objects $X, Y \in \C_M$ we define
$$ \Map_{\C_M}(X, Y) = M(\check{X} \otimes Y) $$
Given a triple of objects $X, Y, Z \in \C_M$ the composition law $$ \Map_{\C_M}(\check{X}, Y) \times \Map_{\C_M}(\check{Y},Z) \lrar \Map_{\C_M}(\check{X},Z) $$
is given by the composition
$$ M(\check{X} \otimes Y) \times M(\check{Y} \otimes Z) \lrar M(\check{X} \otimes Y \otimes \check{Y} \otimes Z) \lrar  M(\check{X} \otimes Z) $$
where the first map is given by the lax symmetric monoidal structure on the functor $M$ and the second is induced by the evaluation map
$$ \ev_Y : \check{Y} \otimes Y \lrar 1 $$
in $\B^{\ev}_1 $.

\item
The symmetric monoidal structure is defined in a straight forward way using the lax monoidal structure of $M$.
\end{enumerate}
It is not hard to see that if $M$ is non-degenerate then $\C_M$ is \textbf{quasi-unital}, i.e. each object contains a morphism which \textbf{behaves} like an identity map (see~\cite{har}). This construction determines a functor
$$ G: \Fun_{\nd}^{\lax}(\B^{\ev}_1,\Grp_{\infty}) \lrar \Cat^{\qu,\otimes}_{\B^{\un}_0 /} $$
where $\Cat^{\qu,\otimes}$ is the $\infty$-category of symmetric monoidal quasi-unital categories (i.e. commutative algebra objects in the $\infty$-category $\Cat^{\qu}$ of quasi-unital $\infty$-categories). In~\cite{har} it is proved that the forgetful functor
$$ S:\Cat \lrar\Cat^{\qu} $$
From $\infty$-categories to quasi-unital $\infty$-categories is an \textbf{equivalence} and so the forgetful functor
$$ S^{\otimes}:\Cat^{\otimes} \lrar \Cat^{\qu,\otimes} $$
is an equivalence as well.

Now recall that
$$ \Cat^{\sur}_{\B^{\ev}_1 /} \subseteq \Cat^{\nd}_{\B^{\ev}_1 /} $$
is the full subcategory spanned by essentially surjective functors $\vphi: \B^{\ev}_1 \lrar \C$. The fiber functor construction $\vphi \mapsto M_\vphi$ induces a functor
$$ F: \Cat^{\sur}_{\B^{\ev}_1 /} \lrar \Fun_{\nd}^{\lax}(\B^{\ev}_1,\Grp_{\infty}) $$

The composition $G \circ F$ gives a functor
$$ \Cat^{\sur}_{\B^{\ev}_1 / } \lrar \Cat^{\qu,\otimes}_{\B^{\un}_0 /} $$

We claim that $G \circ F$ is in fact \textbf{equivalent} to the composition
$$ \Cat^{\sur}_{\B^{\ev}_1 / } \x{T}{\lrar} \Cat^{\otimes}_{\B^{\un}_0 / } \x{S}{\lrar} \Cat^{\qu,\otimes}_{\B^{\un}_0 / } $$
where $T$ is given by the restriction along $X_+:\B^{\un}_0 \hrar \B^{\ev}_1$ and $S$ is the forgetful functor.

Explicitly, we will construct a natural transformation
$$ N:G \circ F \x{\simeq}{\lrar} S \circ T $$
In order to construct $N$ we need to construct for each non-degenerate functor $\vphi: \B^{\ev}_1 \lrar \D$ a natural pointed functor
$$ N_\vphi: \C_{M_\vphi} \lrar \D $$
The functor $N_\vphi$ will map the objects of $\C_{M_\vphi}$ (which are the objects of $\B^{\ev}_1$) to $\D$ via $\vphi$. Then for each $X,Y \in \B^{\ev}_1$ we can map the morphisms
$$ \Map_{\C_{M_{\vphi}}}(X,Y) = \Map_{\D}(1,\check{X} \otimes Y) \lrar \Map_{\D}(X,Y) $$
via the duality structure - to a morphism $f: 1 \lrar \check{X} \otimes Y$ one associates the morphism $\what{f}: X \lrar Y$ given as the composition
$$ X \x{Id \otimes f}{\lrar} X \otimes \check{X} \otimes Y \x{\vphi(\ev_X) \otimes Y}{\lrar} Y $$
Since $\D$ has duals we get that $N_\vphi$ is fully faithful and since we have restricted to essentially surjective $\vphi$ we get that $N_\vphi$ is essentially surjective. Hence $N_\vphi$ is an equivalence of quasi-unital symmetric monoidal $\infty$-categories and $N$ is a natural equivalence of functors.

In particular we have a homotopy commutative diagram:
$$ \xymatrix{
& \Cat^{\sur}_{\B^{\ev}_1 / } \ar_{F}[dl] \ar^{T}[dr] & \\
\Fun_{\nd}^{\lax}(\B^{\ev}_1,\Grp_{\infty}) \ar_{G}[dr] & & \Cat^{\otimes}_{\B^{\un}_0 /} \ar^{S}[dl] \\
& \Cat^{\qu,\otimes}_{\B^{\un}_0 /}  & \\
}$$

Now from Lemma~\ref{0-to-1-ev} we see that $T$ is fully faithful. Since $S$ is an equivalence of $\infty$-categories we get
\begin{cor}\label{retract}
The functor  $G \circ F$ is fully faithful.
\end{cor}

We are now ready to complete the proof of~\ref{cobordism-last-step-2}. Let $\D$ be a symmetric monoidal $\infty$-category with duals and let $\vphi: \B \lrar \D$ be a non-degenerate functor. We wish to show that the space of maps
$$ \Map_{\Cat^{\sur}_{\B^{\ev}_1 /}}(\iota,\vphi) $$
is contractible. Consider the sequence

$$ \Map_{\Cat^{\sur}_{\B^{\ev}_1 /}}(\iota,\vphi) \lrar \Map_{\Fun_{\nd}^{\lax}(\B^{\ev}_1,\Grp_{\infty})}(M_\iota,M_\vphi) \lrar
\Map_{\Cat^{\qu,\otimes}_{\B^{\un}_0 /}}(\B^{\ori}_1,\D) $$
By Theorem~\ref{qu-cobordism} the middle space is contractible and by lemma~\ref{retract} the composition
$$ \Map_{\Cat^{\sur}_{\B^{\ev}_1 /}}(\iota,\vphi) \lrar \Map_{\Cat^{\qu,\otimes}_{\B^{\un}_0 /}}(\B^{\ori}_1,\D) $$
is a weak equivalence. Hence we get that
$$ \Map_{\Cat^{\sur}_{\B^{\ev}_1 /}}(\iota,\vphi) $$
is contractible. This completes the proof of Theorem~\ref{cobordism-last-step-2}.


\begin{thebibliography}{cobordism}

\bibitem[BaDo]{bd}
Baez, J., Dolan, J., Higher-dimensional algebra and topological qauntum field theory, Journal of Mathematical Physics, 36 (11), 1995, 6073--6105.

\bibitem[Har]{har}
Harpaz, Y. Quasi-unital $\infty$-categories, PhD Thesis.

\bibitem[Lur1]{lur1}
Lurie, J., On the classification of topological field theories, Current Developments in
Mathematics, 2009, p. 129-–280, \url{http://www.math.harvard.edu/~lurie/papers/cobordism.pdf}.

\bibitem[Lur2]{lur2}
Lurie, J. \emph{Higher Algebra}, \url{http://www.math.harvard.edu/~lurie/papers/higheralgebra.pdf}.

\bibitem[Lur3]{lur3}
Lurie, J., \emph{Higher Topos Theory}, Annals of Mathematics Studies, 170, Princeton University
Press, 2009, \url{http://www.math.harvard.edu/~lurie/papers/highertopoi.pdf}.




\bibitem[Kal]{kal}
Kaledin, D., Homological methods in non-commutative geometry, preprint, \url{http://imperium.lenin.ru/~kaledin/math/tokyo/final.pdf}

\end{thebibliography}
\end{document}